\documentclass{amsart}
\usepackage{lineno,amsmath,amsthm,amssymb,fancybox,fancyhdr,multirow,ulem,tikz-cd,caption,longtable,color}

\newtheorem{theorem}{Theorem}%[section]
\newtheorem{lemma}[theorem]{Lemma}%[section]
\newtheorem{definition}[theorem]{Definition}%[section]
\newtheorem{corollary}[theorem]{Corollary}%[section]
\newtheorem{remark}[theorem]{Remark}%[section]
\newcommand{\Z}{{\mathbb Z}}
\newcommand{\Q}{{\mathbb Q}}
\newcommand{\R}{{\mathbb R}}
\newcommand{\C}{{\mathbb C}}
\newcommand{\F}{{\mathbb F}}
\newcommand{\Gal}{{\rm Gal}}

\newcommand{\Cl}{{\rm Cl}}
\newcommand{\Clodd}{{\rm Cl}_o}
\newcommand{\legendre}[2]{\genfrac{(}{)}{}{}{#1}{#2}}

\newcommand{\bimax}[4]{$#1=#2$ for $K=\Q(\sqrt{#3},\sqrt{#4})$}
\newcommand{\trimax}[5]{$#1=#2$ for $K=\Q(\sqrt{#3},\sqrt{#4},\sqrt{#5})$}

\title[Imaginary multiquadratic number fields of exponent $3$ and $5$]{Imaginary multiquadratic number fields
with class group of exponent $3$ and $5$}
\author{J\"urgen Kl\"uners} \address{Universit\"at Paderborn, Institut f\"ur 
  Mathematik, Warburger Str. 100, 33098 Paderborn, Germany}

\email{klueners@math.uni-paderborn.de}

\author{Toru Komatsu}
\address{Tokyo University of Science, Faculty of Science and Technology, Department of Mathematics,
  2641 Yamazaki, Noda-shi, Chiba-ken 278-8510, Japan}
\email{komatsu@ma.noda.tus.ac.jp}
\date{}
\begin{document}

\begin{abstract}
In this paper we obtain 
a complete list of imaginary $n$-quadratic fields
with class groups of exponent $3$ and $5$
under ERH for every positive integer $n$
where an $n$-quadratic field is a number field of degree $2^n$
represented as the composite of $n$ quadratic fields.
\end{abstract}

\maketitle

\section{Introduction}
The classification, of number fields with given invariants has a long history, going back at least to Gauss. Brown and Parry \cite{BrownParry} determine
a complete list of imaginary biquadratic fields
with class number $1$.
Yamamura \cite{Yamamura} gives
a complete list of imaginary abelian fields
with class number $1$.
Jung and Kwon \cite{JungKwon} show
a complete list of imaginary biquadratic fields
of class number $3$.
Elsenhans, Kl\"uners and Nicolae \cite[Theorem 1]{EKN}
present a complete list of imaginary quadratic fields
with class groups of exponent $E$
for every $E\leq 5$ and $E=8$ under the extended
Riemann hypothesis (ERH).
We say that $K$ is an $n$-quadratic field
if $K$ is a Galois extension of $\Q$
with $\Gal(K/\Q)\simeq C_2^n$ where
$C_2$ is the cyclic group of order $2$.
In this paper we obtain 
a complete list of imaginary $n$-quadratic fields
with class groups of exponent $3$ or $5$,
that is,
isomorphic to direct products $C_u^r$
of the cyclic group $C_u$ of order $u$
with $u=3,5$ and positive integers $r$, 
under ERH for every positive integer $n$.
We call a $2$-quadratic (resp.\ $3$-quadratic) field
as a biquadratic (resp.\ triquadratic) field.
In this paper we show

\begin{theorem}
\label{maintheorem}
(1)
There exist at least $163,122,32$ and $1$ imaginary biquadratic fields
with class groups isomorphic to $C_3,C_3^2,C_3^3$ and $C_3^4$,
respectively.\newline
(2)
There exist at least $23,29,7$ and $1$ imaginary triquadratic fields
with class groups isomorphic to $C_3,C_3^2,C_3^3$ and $C_3^4$,
respectively.\newline
(3)
Under ERH,
there exist no imaginary biquadratic and triquadratic fields
with class groups of exponent $3$ other than those in (1) and (2).
\end{theorem}

\begin{theorem}
\label{maintheoremFive}
(1)
There exist at least $243,274,54$ and $1$ imaginary biquadratic fields
with class groups isomorphic to $C_5,C_5^2,C_5^3$ and $C_5^4$,
respectively.\newline
(2)
There exist at least $18,26,6$ and $1$ imaginary triquadratic fields
with class groups isomorphic to $C_5,C_5^2,C_5^3$ and $C_5^4$,
respectively.\newline
(3)
Under ERH,
there exist no imaginary biquadratic and triquadratic fields
with class groups of exponent $5$ other than those in (1) and (2).
\end{theorem}

\begin{remark}
\label{greaterthanthree}
{\rm
It is known that,
for every $n\geq 4$,
there exist no imaginary $n$-quadratic fields
with odd class numbers
(cf.\ Theorem~\ref{FroehlichOddCri}
\cite[Theorem 5.2]{Froehlich} below).
}
\end{remark}

%\begin{remark}
%\label{assumptionFive}
%{\rm
%Under ERH and a certain assumption (*) before Lemma \ref{TwoBFive},
%there exist no imaginary biquadratic and triquadratic fields
%with class groups of exponent $5$ other than those in (1) and (2)
%of Theorem \ref{maintheoremFive}.
%}
%\end{remark}

The computed tables of fields and some programs can be found on the web page
https://math.uni-paderborn.de/en/ag/ca/research/exponent/.

\section{Decomposition with prime power conductor subfields}

In order to study the $2$-part of class groups
of $n$-quadratic fields,
let us introduce a symbol $[x,y]$
defined in Fr\"ohlich's book \cite[\S 4]{Froehlich}.
In \cite{Froehlich} he defines the symbol as an element in $\Z_{\ell}$
and uses it as an element in $\F_{\ell}$ to apply the theorems
where $\Z_{\ell}$ is the ring of $\ell$-adic integers
and $\F_{\ell}$ is the field of $\ell$ elements
for a prime number $\ell$.
In this paper we define the symbol $[x,y]$
as an element in $\F_{\ell}$ with $\ell=2$ from the beginning.
Let $S$ be a finite set of prime numbers,
and put $S'=S\cup\{-1\}$ if $2\in S$,
and $S'=S$ otherwise.
For a pair $(x,y)\in S'\times S$
with $x\not=y$ and $(x,y)\not=(-1,2)$
we define the symbol $[x,y]\in\F_2$ as follows.
Let $\overline{a}$ denote the image of an $a\in\Z$
under the canonical map $\Z\to \Z/2\Z\simeq \F_2$.
For an odd prime number $p$
let us fix a primitive $(p-1)$-st root of unity $w_p$
in the $p$-adic number field $\Q_p$. 
For $p\ne y\in S$ there exists
a rational integer $r_{p,y}$ such that
\begin{center}
$w_p^{r_{p,y}}\equiv y\pmod{p}$.
\end{center}
We define $[p,y]\in\F_2$ by $[p,y]=\overline{r_{p,y}}$
for such an integer $r_{p,y}$.
For the case $p=2$ we put $w_2=5$.
For $2\ne y\in S$ there exist
rational integers $r_{2,y}$ and $r_{-1,y}$ such that
\begin{center}
$w_2^{r_{2,y}}(-1)^{r_{-1,y}}\equiv y\pmod{8}$.
\end{center}
We define $[2,y]$ and $[-1,y]$ in $\F_2$
by $[2,y]=\overline{r_{2,y}}$ and
$[-1,y]=\overline{r_{-1,y}}$
for such integers $r_{2,y}$ and $r_{-1,y}$, respectively.
It is easy to see the following relation
between $[x,y]$ and the Legendre symbol
$\left(\dfrac{a}{p}\right)$. Note that the value $[-1,2]$ is not defined.

\begin{lemma}
\label{relationtoLegendre}
Let $[x,y]\in S'\times S$ be given with $x\ne y$. Then
\[ (-1)^{[x,y]} = \begin{cases}
\legendre{y}{x} & \mbox{if }x\notin\{-1,2\}\\
\legendre{x}{y} & \mbox{if }x \in\{-1,2\}, y\ne 2
\end{cases}.
\]
\end{lemma}

For a number field $K$ of finite degree
let $\Cl(K)$ and $\Cl_+(K)$ denote
the class group and the narrow class group of $K$,
and $h(K)$ and $h_+(K)$
the class number and the narrow class number of $K$,
respectively.
Note that $\Cl(K)=\Cl_+(K)$ and $h(K)=h_+(K)$
if $K$ is totally imaginary.
For a number field $K$
let $G_+(K/\Q)$ denote
the narrow genus field of $K$ over $\Q$,
that is,
the maximal extension of $K$
which is unramified at all finite primes,
and is the composite $KL$ of $K$ and some
abelian extension $L$ of $\Q$.
For an abelian field $K$
let $\nu(K)$ denote the minimal number
of generators of $\Gal(K/\Q)$.
Note that $\nu(K)=n$ for an $n$-quadratic field $K$. Before we state the theorem we need a definition of a function.
\begin{definition}
	Let $K$ be an $n$-quadratic field not containing the cyclotomic field $\Q(\zeta_8)$, $p_1$ be the smallest prime number ramified in $K$ and $p>p_1$ be another prime number ramified in $K$. Then we define:
	\[a_{p_1}(p):=
	\begin{cases}
	[p_1,p] &\text{if $p_1\not=2$ or $\sqrt{2}\in K$,}\\ 
	{[-1,p]} &\text{if $\sqrt{-1}\in K$,}\\
	{[2,p]}+{[-1,p]} &\text{if $\sqrt{-2}\in K$.}\\
	\end{cases}\]
\end{definition}
Note that the condition $\Q(\zeta_8)\not\subseteq K$ implies that exactly one condition of the three different cases is true.

\begin{theorem}[Fr\"ohlich {\cite[Theorem 5.2]{Froehlich}}]
\label{FroehlichOddCri}
Let $K$ be an $n$-quadratic field,
and $S$ the set of prime numbers ramifying in $K$.\newline
(1)
If $h_+(K)$ is odd,
then $K=G_+(K/\Q)$.\newline
(2)
Assume $K=G_+(K/\Q)$.
Then $h_+(K)$ is odd if and only if
one of the following conditions (2.1),(2.2) and (2.3) holds:
\newline
(2.1)
$n=1$.\newline
(2.2)
$n=2$, and (i) or (ii) holds:
\begin{enumerate}
\item[(i)]
$S=\{2\}$, that is, $K=\Q(\sqrt{-1},\sqrt{-2})$,
\item[(ii)]
$S=\{p_1,p_2\}$ with distinct prime numbers $p_1<p_2$
such that
$[p_2,p_1]\not=0$ or $a_{p_1}(p_2)\not=0$.
\end{enumerate}
(2.3)
$n=3$, and (i) or (ii) holds:
\begin{enumerate}
\item[(i)]
$S=\{2,p_2\}$ with an odd prime number $p_2$ such that $[p_2,2]\not=0$,
\item[(ii)]
$S=\{p_1,p_2,p_3\}$ with $p_1<p_2<p_3$
such that $\det M\not=0$ where
\begin{center}
$M=
\begin{pmatrix}
[p_2,p_1] & 0 & [p_3,p_1]\\
a_{p_1}(p_2) & [p_3,p_2] & 0\\
0 & [p_2,p_3] & a_{p_1}(p_3)\\
\end{pmatrix}
$.
\end{center}
\end{enumerate}
\end{theorem}

\begin{remark}
\label{caseABCD}
{\rm 
(1) For an $n$-quadratic field $K$,
the cases (A), (B), (C) and (D)
at \S 5 in \cite{Froehlich} mean that
$\Q(\zeta_8)\cap K$ are equal to
$\Q(\zeta_8),\Q(\sqrt{-2}),\Q(\sqrt{-1}),$ and
$\Q(\sqrt{2})$, respectively,
where $\zeta_8$ is a primitive $8$-th root of unity in $\C$.
In our situation,
the invariant $\lambda$ for the case (B) is equal to $1$.\newline
(2) We have $[x,y]=-[x,y]$ since $[x,y]$ is an element in $\F_2$.
For general $\ell$,
the definition of $a_{p_1}(p)$ for the third case
is $a_{p_1}(p)={[2,p]}-{[-1,p]}$
in Theorem~\ref{FroehlichOddCri} (2.2) (ii).
The same adjustment occurs at the main diagonal of $M$
in Theorem~\ref{FroehlichOddCri} (2.3) (ii).
}
\end{remark}

\begin{theorem}[Fr\"ohlich {\cite[Theorem 2.15]{Froehlich}}]
\label{FroehlichEquiv}
For an abelian field $K$, the following conditions (i) and (ii)
are equivalent:
\begin{itemize}
\item[(i)]
$K=G_+(K/\Q)$.\quad
\item[(ii)]
$K$ is the composite of fields with prime power conductors.
\end{itemize}
\end{theorem}

As in the paper \cite{Uchida}
we say that $K$ is a field of type I
if $K$ is an imaginary abelian field with the conditions
of Theorem~\ref{FroehlichEquiv}.
For a field $K$ of type I
we say that $K=K_1\cdots K_r$
is the decomposition of $K$
into prime power conductor fields,
or simply the decomposition of $K$
where $K_i$ are subfields of $K$
of prime power conductors $p_i^{e_i}$
with distinct prime numbers $p_i$
and positive integers $e_i$.

As a corollary of Theorem~\ref{FroehlichOddCri} we have

\begin{corollary}
\label{nonIRR}
Let $K$ be a field of type {\rm I}
with decomposition $K=K_1K_2K_3$.
If $K_1$ is an imaginary quadratic field,
and $K_2$ and $K_3$ are real quadratic fields,
then $2\mid h(K)$.
\end{corollary}

\begin{proof}
We consider the following four essential cases.
\newline
Case 1:
$K=\Q(\sqrt{-p_1},\sqrt{p_2},\sqrt{p_3})$
with $p_1\equiv 3\pmod{4}$ and $p_2\equiv p_3\equiv 1\pmod{4}$.
It follows from $p_2\equiv 1\pmod{4}$ that
$a_{p_1}(p_2)=[p_1,p_2]=[p_2,p_1]$ and $[p_2,p_3]=[p_3,p_2]$.
 $a_{p_1}(p_3)=[p_1,p_3]=[p_3,p_1]$ since $p_3\equiv 1\pmod{4}$.
Thus, if $[p_2,p_1],[p_3,p_2]$ or $[p_3,p_1]$
is equal to $0$,
then the first, second or third column of $M$ is the zero vector,
which implies $\det M=0$.
When $[p_2,p_1]=[p_3,p_2]=[p_3,p_1]=1$,
the sum of the three columns of $M$ is equal to the zero vector,
which yields $\det M=0$.
\newline
Case 2:
$K=\Q(\sqrt{2},\sqrt{p_2},\sqrt{-p_3})$
with $p_1=2$, $p_2\equiv 1\pmod{4}$ and $p_3\equiv 3\pmod{4}$.
It follows from $[2,p]=[p,2]$ that
$a_{p_1}(p_2)=[2,p_2]=[p_2,p_1]$
and $a_{p_1}(p_3)=[2,p_3]=[p_3,p_1]$.
By $p_2\equiv 1\pmod{4}$
we have $[p_2,p_3]=[p_3,p_2]$.
In the same way as in case 1,
we see $\det M=0$.
\newline
Case 3:
$K=\Q(\sqrt{-1},\sqrt{p_2},\sqrt{p_3})$
with $p_1=2$ and $p_2\equiv p_3\equiv 1\pmod{4}$.
Then we have $a_{p_1}(p_2)=a_{p_1}(p_3)=0$.
Thus the first, the third column of $M$ or their sum
is equal to the zero vector, which implies $\det M=0$.
\newline
Case 4:
$K=\Q(\sqrt{-2},\sqrt{p_2},\sqrt{p_3})$
with $p_1=2$ and $p_2\equiv p_3\equiv 1\pmod{4}$.
Then we have
$a_{p_1}(p_2)=[2,p_2]+[-1,p_2]=[2,p_2]=[p_2,p_1]$,
$[p_2,p_3]=[p_3,p_2]$
and $a_{p_1}(p_3)=[2,p_3]+[-1,p_3]=[2,p_3]=[p_3,p_1]$.
In the same way as in case 1,
we see $\det M=0$.
\end{proof}

\begin{remark}
\label{UchidaLemma}
{\rm
Uchida shows Corollary~\ref{nonIRR}
in a more general case \cite[Lemma 1]{Uchida}, that is,
if $K_1$ is an imaginary cyclic field of 2-power degree,
and $K_2$ and $K_3$ are real quadratic fields,
then $2\mid h(K)$.
}
\end{remark}

Let us consider the odd part $\Clodd(K)$ of $\Cl(K)$,
that is, the subgroup of $\Cl(K)$
consisting of all the classes of odd order.
The following well-known theorem is useful
to study $\Clodd(K)$ for an $n$-quadratic field $K$.

\begin{theorem}[Lemmermeyer {\cite[(1.1)]{Lemmer}}]
\label{LemmerOdd}
Let $K/k$ a Galois extension of number fields with $\Gal(K/k)\simeq V_4$
where $V_4$ is Klein's four-group.
Let $k_1,k_2$ and $k_3$ stand for
the three intermediate fields of $K/k$
with $[K:k_j]=[k_j:k]=2$.
Then we have
\[\Clodd(K)\simeq \Clodd(k)\times\prod_{j=1}^3
\Clodd(k_j)/\Clodd(k).\]
\end{theorem}

We define $p^*:=(-1)^{(p-1)/2}p$ for odd prime numbers $p$. 
Let $P^*=\{8,-4,-8\}\cup\{p^*\,\vert\,p \text{ an odd prime number}\}$,
$P^*_+=\{p^*\in P^*\,\mid\,p^*>0\}$
and
$P^*_-=\{p^*\in P^*\,\mid\,p^*<0\}$.
Let $E(K)$ represent the exponent of $\Cl(K)$ for a number field $K$.

\begin{theorem}
\label{allforms}
Fix an odd number $u>0$.
Let $K$ be an imaginary $n$-quadratic field
such that $E(K)\mid u$.
Then $K$ is of one of the following forms 
(1), (2) and (3).
\begin{itemize}
\item[(1)]
$K=\Q(\sqrt{p^*})$ 
with $p^*\in P^*_-$
such that $E(\Q(\sqrt{p^*}))\mid u$.
\item[(2)]
$K=\Q(\sqrt{p_1^*},\sqrt{p_2^*})$
with $p_1^*\in P^*_-$ and $p_2^*\in P^*$
such that $\Q(\sqrt{p_1^*})$ is of the form (1) above
and the following (2a) or (2b) holds:
\begin{itemize}
\item[(2a)]
$p_2^*<0$ and 
$E(\Q(\sqrt{p_2^*}))\mid u$, i.e.,
$\Q(\sqrt{p_2^*})$ is also of the form (1),
\item[(2b)]
$p_2^*>0$ and
$E(\Q(\sqrt{p_1^*p_2^*}))\mid 2u$.
\end{itemize}
\item[(3)]
$K=\Q(\sqrt{p_1^*},\sqrt{p_2^*},\sqrt{p_3^*})$
with $p_1^*,p_2^*\in P^*_-$, $p_3^*\in P^*$
such that the
three subfields
$\Q(\sqrt{p_1^*},\sqrt{p_2^*}),
\Q(\sqrt{p_1^*},\sqrt{p_3^*})$
and
$\Q(\sqrt{p_2^*},\sqrt{p_3^*})$
are of the form (2).
\end{itemize}
\end{theorem}

\begin{proof}
If $n=1$, then the case (1) is obvious from genus theory.
Assume $n=2$.
Theorems~\ref{FroehlichOddCri} and \ref{FroehlichEquiv}
imply that $K=\Q(\sqrt{p_1^*},\sqrt{p_2^*})$
for some $p_1^*\in P^*_-$ and $p_2^*\in P^*$.
When $p_1=p_2$, we have $p_1=p_2=2$,
that is, $K=\Q(\sqrt{-4},\sqrt{-8})$,
for which (2) holds.
Under $p_1\not=p_2$,
since the extension
$K/\Q(\sqrt{p_1^*})$
is ramified at $p_2$,
the norm map $\Cl(K)
\to \Cl(\Q(\sqrt{p_1^*}))$ is surjective
(cf.\ \cite[Theorem 10.1]{Washington}).
This means that 
$E(\Q(\sqrt{p_1^*}))\mid E(K)$,
which implies that $\Q(\sqrt{p_1^*})$ is of the form (1).
In the same way, if $p_2^*<0$, then 
$\Q(\sqrt{p_2^*})$ is of the form (1).
When $p_2^*>0$,
put $k=\Q(\sqrt{p_1^*p_2^*})$.
Theorem~\ref{LemmerOdd} shows that $E(k)\mid 2^ru$
for some rational integer $r\geq 0$.
It follows from genus theory that $2\mid E(k)$.
Suppose $4\mid E(k)$.
Then $k$ has an unramified extension $M_k$
with $\Gal(M_k/k)\simeq C_4$.
Note that $M_k$ is not contained in $K$
since $\Gal(K/\Q)$ has exponent $2$.
Thus the lift $M_kK/K$ of $M_k/k$ to $K$
is nontrivial, that is, $2\mid h(K)$ by class field theory.
This is a contradiction.
Thus we have $E(k)\mid 2u$.
For the case $n=3$, 
Theorems~\ref{FroehlichOddCri} and \ref{FroehlichEquiv}
imply that $K=\Q(\sqrt{p_1^*},\sqrt{p_2^*},\sqrt{p_3^*})$
for some $p_1^*\in P^*_-$ and $p_2^*,p_3^*\in P^*$.
Due to Corollary~\ref{nonIRR},
we have $p_2^*<0$ or $p_3^*<0$, say, $p_2^*\in P^*_-$.
In the same way as for $n=2$,
due to the surjectivity of the norm maps,
the imaginary subfields 
$\Q(\sqrt{p_1^*},\sqrt{p_2^*}),
\Q(\sqrt{p_1^*},\sqrt{p_3^*})$
and
$\Q(\sqrt{p_2^*},\sqrt{p_3^*})$
have class groups of exponent dividing $u$,
and thus they are of the form (2).
\end{proof}

Let $\mathcal{K}_{\rm 1},\mathcal{K}_{\rm 2}$
and $\mathcal{K}_{\rm 3}$ (depending on $u$)
denote the families of all fields 
of the forms (1), (2) and (3), respectively. 
As a corollary of Theorem~\ref{allforms} we have

\begin{corollary}
\label{KtwoimpliesKthree}
If $\mathcal{K}_{\rm 2}$ is finite for a certain choice of $u$,
then so is $\mathcal{K}_{\rm 3}$.
\end{corollary}

\begin{remark}
\label{relationtoUchida}
{\rm
Our inductive method in Theorem~\ref{allforms}
with $u=1$ coincides
with that in the proof of Proposition 5 in \cite{Uchida}.
}
\end{remark}

\section{Finiteness of the families under ERH and ideas for the computation}

\begin{theorem}[Boyd and Kisilevsky {\cite[Theorem 4]{BoydKisi}}]
\label{BoydKisiERH}
Let $k$ be an imaginary quadratic field of discriminant $D<0$.
Then, under the extended Riemann hypothesis (ERH),
for any $\eta>0$,
we have $E(k)>(\log|D|)/((2+\eta)\log\log|D|)$
for sufficiently large $|D|$.
\end{theorem}

\begin{corollary}
\label{finiteOneTwo}
Under ERH,
the families $\mathcal{K}_{\rm 1}$ and
$\mathcal{K}_{\rm 2}$ are finite for any odd $u$.
\end{corollary}

\begin{proof}
It follows directly from Theorem~\ref{BoydKisiERH}
that $\mathcal{K}_{\rm 1}$ is finite.
Assume $K\in\mathcal{K}_{\rm 2}$,
that is,
$K=\Q(\sqrt{p_1^*},\sqrt{p_2^*})$
where $p_1^*\in P^*_-$ and $p_2^*\in P^*$
satisfy
$E(\Q(\sqrt{p_1^*}))\mid u$
and
$E(\Q(\sqrt{p_1^{*c}p_2^*}))\mid 2^cu$
where $c=0$ if $p_2^*<0$ and $c=1$ otherwise.
Let $\mathcal{F}$ denote
the family consisting of all
imaginary quadratic fields $k$ with $E(k)\mid 2u$.
Theorem~\ref{BoydKisiERH} yields that
$\mathcal{F}$ is finite under ERH.
Since $\Q(\sqrt{p_1^*})$
and $\Q(\sqrt{p_1^{*c}p_2^*})$ belong to 
the finite family $\mathcal{F}$,
the set of such pairs $(p_1^*,p_2^*)$ is also finite.
Hence $\mathcal{K}_{\rm 2}$ is finite.
\end{proof}

\begin{corollary}
\label{finiteMulti}
For a fixed odd number $u>0$,
under ERH,
there exist finitely many imaginary $n$-quadratic fields $K$
such that $E(K)\mid u$.
\end{corollary}

Corollaries~\ref{finiteOneTwo} and \ref{finiteMulti}
with $u=3$ hold without ERH since we have

\begin{theorem}[Heath-Brown {\cite[Theorems 1 and 2]{HB}}]
\label{HBthm}
If $E$ is equal to $5$, $2^m$ or $2^m3$ with a rational integer $m\geq 0$,
then there exists an ineffective constant $d_E$ such that
$E(\Q(\sqrt{-d}))\not=E$
for every fundamental discriminant $-d$ with $d>d_E$.
\end{theorem}

\begin{remark}
\label{ineffective}
{\rm
Unfortunately, the constants $d_E$ in Theorem~\ref{HBthm}
are not effective.
Therefore we need to apply ERH for the explicit computation
of finding all such quadratic fields.
}
\end{remark}
In order to explain our algorithm to compute those families for a given odd $u$ we would like to define some subfamilies of $\mathcal{K}_{\rm 2}$ and $\mathcal{K}_{\rm 3}$, respectively. Under ERH,
let $\kappa_{\rm 1},\kappa_{\rm 2}$ and $\kappa_{\rm 3}$
denote the cardinalities of
the finite families
$\mathcal{K}_{\rm 1},\mathcal{K}_{\rm 2}$ and $\mathcal{K}_{\rm 3}$.

Let $I_1$ be the set consisting of 
the discriminants fields in $\mathcal{K}_1$, i.e.
\begin{eqnarray}\label{pstar}
I_1:=&\{p^*\in P^*_-\,\mid\,E(\Q(\sqrt{p^*})
\text{ divides $u$}\}.
\end{eqnarray}
For each $p^*\in I_1$
let us define a subset $R_{p^*}$ of $P^*_+$ by
\begin{eqnarray}\label{Rpstar}
R_{p^*}=&\{q^*\in P^*_+\,\mid\,E(\Q(\sqrt{p^*q^*}))
\text{ is a divisor of $2u$ and $p\not=q$}\}.
\end{eqnarray}
By definition, we have that
$\mathcal{K}_{\rm 2} =\mathcal{K}_{\rm 2a}\cup\mathcal{K}_{\rm 2b}$
where
\begin{eqnarray}\label{eq2a}
\mathcal{K}_{\rm 2a}
=&\{\Q(\sqrt{p^*},\sqrt{q^*})\,\mid\,p^*,q^*\in I_1,|p^*|< |q^*|\},\\
\mathcal{K}_{\rm 2b}\label{eq2b}
=&\{\Q(\sqrt{p^*},\sqrt{q^*})\,\mid\,p^*\in I_1,q^*\in R_{p^*}\}.
\end{eqnarray}
Here we treat $\Q(\sqrt{-4},\sqrt{-8})$ as
a member of $\mathcal{K}_{\rm 2a}$, not of $\mathcal{K}_{\rm 2b}$.
Let $\kappa_{\rm 2a}$ and $\kappa_{\rm 2b}$ denote
the cardinalities of
$\mathcal{K}_{\rm 2a}$ and $\mathcal{K}_{\rm 2b}$,
respectively.
Let $\mathcal{K}_{\rm 3a}$
and 
$\mathcal{K}_{\rm 3b}$
denote the subfamilies of $\mathcal{K}_{\rm 3}$ with
$\mathcal{K}_{\rm 3}=\mathcal{K}_{\rm 3a} \cup \mathcal{K}_{\rm 3b}$
such that
\begin{eqnarray}
%$\begin{array}{r@{\,}l}
\mathcal{K}_{\rm 3a}\label{eq3a}
=&\{\Q(\sqrt{p_1^*},\sqrt{p_2^*},\sqrt{p_3^*})\in\mathcal{K}_{\rm 3}
\,\mid\,p_1^*,p_2^*,p_3^*\in P^*_-\},\\
\mathcal{K}_{\rm 3b}\label{eq3b}
=&\{\Q(\sqrt{p_1^*},\sqrt{p_2^*},\sqrt{p_3^*})\in\mathcal{K}_{\rm 3}
\,\mid\,p_1^*,p_2^*\in P^*_-, p_3^{*}\in P^*_+\}.
\end{eqnarray}
%\end{center}
with cardinalities $\kappa_{\rm 3a}$ and $\kappa_{\rm 3b}$, respectively.
The fields $\Q(\sqrt{-4},\sqrt{-8},\sqrt{p^*})$ 
with $p^*\in P^*_-$ (resp.\ $p^*\in P^*_+$)
are treated
as members of $\mathcal{K}_{\rm 3a}$
(resp.\ $\mathcal{K}_{\rm 3b}$).

We know for a given odd $u$ that our five subfamilies $\mathcal{K}_1, \mathcal{K}_{2a}, \mathcal{K}_{2b}, \mathcal{K}_{3a},$ and  $\mathcal{K}_{3b}$ are finite assuming ERH. The family $\mathcal{K}_1$ or equivalently $I_1$ was computed in \cite{EKN} for $u=3,5$. Note that the fields in $\mathcal{K}_2$ can be easily computed knowing $I_1$. For each $p^*, q^*\in I_1$  with $p^*\ne q^*$
we can determine whether $E(\Q(\sqrt{p^*},\sqrt{q^*}))\mid u$
or not, using the structures of its quadratic subfields
due to Theorem~\ref{FroehlichOddCri} for the oddness
and Theorem~\ref{LemmerOdd} for the explicit odd part, in particular,
\begin{center}
	$\Clodd(\Q(\sqrt{p^*},\sqrt{q^*}))\simeq
	\Clodd(\Q(\sqrt{p^*}))
	\times\Clodd(\Q(\sqrt{q^*}))
	\times\Clodd(\Q(\sqrt{p^*q^*}))$.
\end{center}
Therefore, we only need to control the odd part of the class group of $\Q(\sqrt{p^*q^*})$. Altogether, we have to apply these tests to all subsets of order 2 of $I_1$.

After assuming the knowledge of $I_1$ the most time consuming part of our algorithm is to compute the set $R_p^*$ for each prime $p^*\in I_1$.
We describe the computation of this set in Section \ref{lastsec}. For $u=3$ we also give a shortcut in the following section. For each $p^*\in I_1$ and $q^*\in R_{p^*}$,
we can check whether $E(\Q(\sqrt{p^*},\sqrt{q^*}))\mid u$
or not, in the same way as for $p^*,q^*\in I_1$ 
due to Theorems~\ref{FroehlichOddCri} and \ref{LemmerOdd}.

In order to compute $\mathcal{K}_{3a}$ we have to do a clever search in the set $\mathcal{K}_{2a}$. More precisely, by Theorem \ref{allforms} we have to find all triples $p_1^*,p_2^*,p_3^*\in P^*_-$ such that the fields 
$k_{12}:=\Q(\sqrt{p_1^*},\sqrt{p_2^*}),
k_{13}:=\Q(\sqrt{p_1^*},\sqrt{p_3^*})$
and
$k_{23}:=\Q(\sqrt{p_2^*},\sqrt{p_3^*})$ are contained in $\mathcal{K}_{2a}$. Define  $K:=\Q(\sqrt{p_1^*},\sqrt{p_2^*},\sqrt{p_3^*})$ to be one of those fields. We can determine whether $E(K)\mid u$
or not, using the structures of its quadratic subfields
due to Theorem~\ref{FroehlichOddCri} for the oddness
and Theorem~\ref{LemmerOdd} for the explicit odd part, in particular,
$\Clodd(K)$ is isomorphic to
\begin{center}
	$\begin{array}{r@{\,}l}
	&\Clodd(k_1)
	\!\times\! (\Clodd(k_1k_2)/\Clodd(k_1))
	\!\times\! (\Clodd(k_1k_3)/\Clodd(k_1))
	\!\times\! (\Clodd(k_1k_{23})/\Clodd(k_1))\\
	\simeq&\Clodd(k_1)
	\!\times\! (\Clodd(k_2)\!\times\!\Clodd(k_{12}))
	\!\times\! (\Clodd(k_3)\!\times\!\Clodd(k_{13}))
	\!\times\! (\Clodd(k_{23})\!\times\!\Clodd(k_{123}))\\
	\simeq&\Clodd(k_1)
	\times\Clodd(k_2)\times\Clodd(k_{3})
	\times\Clodd(k_{12})\times\Clodd(k_{13})
	\times\Clodd(k_{23})\times\Clodd(k_{123}),
	\end{array}$
\end{center}
where 
$k_{123}:=\Q(\sqrt{p_{1}^*p_{2}^*p_{3}^*})$. For the odd part of the class group we only need to compute the class group of $k_{123}$ because the other six quadratic class groups have already the right shape since $k_{12},k_{13},k_{23}\in \mathcal{K}_{2a}$. Furthermore, we have to check the condition of Theorem \ref{FroehlichOddCri}.

The story is very similar for the computation of $\mathcal{K}_{3b}$. Here we have to find all $p_1^*,p_2^*\in P^*_-$, and $p_3^*\in P^*_+$ such that the field
 $k_{12}:=\Q(\sqrt{p_1^*},\sqrt{p_2^*})$ is contained in $\mathcal{K}_{2a}$ and the fields
$k_{13}:=\Q(\sqrt{p_1^*},\sqrt{p_3^*})$
and
$k_{23}:=\Q(\sqrt{p_2^*},\sqrt{p_3^*})$ are contained in $\mathcal{K}_{2b}$.
For each of those fields $K:=\Q(\sqrt{p_1^*},\sqrt{p_2^*},\sqrt{p_3^*})$ we can apply the same tests using Theorems~\ref{FroehlichOddCri} and \ref{LemmerOdd} as in the case for the family $\mathcal{K}_{3a}$.

Note that we use ERH for the computations of $\mathcal{K}_1$ and $R_p^*$. The other families are depending on those sets.

\section{Computation for the case $u=3$ under ERH}

In this section we fix $u=3$. 
\begin{theorem}[Baker \cite{Baker}, Stark \cite{Stark}]
%Watkins \cite{Watkins}]
\label{CLone}
Let $d>0$ be a squarefree rational integer.
Then $\Q(\sqrt{-d})$ has class number $1$
if and only if
$d$ is equal to
$1, 2, 3, 7, 11, 19, 43, 67$ or $163$.
%$d\in\{1, 2, 3, 7, 11, 19, 43, 67, 163\}$.
\end{theorem}

\begin{theorem}[Elsenhans, Kl\"uners and Nicolae
{\cite[Theorems 1 and 2]{EKN}}]
\label{EKNexpthree}
The maximal absolute value $|D|$ of the discriminants $D$
of all the imaginary quadratic fields with class group
of exponent $3$
under the condition $|D|<3.1\times 10^{20}$ is $|D|=4027$.
Under the extended Riemann hypothesis (ERH),
the maximal absolute value $|D|$
of such fields without the condition $|D|<3.1\times 10^{20}$
is equal to $4027$.
\end{theorem}

Theorems~\ref{CLone} and \ref{EKNexpthree} yield

\begin{corollary}[Elsenhans, Kl\"uners and Nicolae
{\cite[Theorems 1 and 2]{EKN}}]
\label{Konelist}
Under ERH, the following is the list of all 
squarefree rational integers $d>0$
such that $\Cl(\Q(\sqrt{-d}))\simeq C_3^r$
for some rational integers $r\geq 0$.
\begin{center}
$\begin{array}{|c|l|c|}
\hline
r& d & \sharp\\
\hline
0&1, 2, 3, 7, 11, 19, 43, 67, 163 & 9\\ \hline
1&23, 31, 59, 83, 107, 139, 211, 283, 307, 331, 379,
 499, 547, 643, 883, 907 & 16\\ \hline
2&4027 & 1\\ \hline
\end{array}$
\end{center}
Here the numbers in the right column $\sharp$
are the numbers of $d$'s in the corresponding row.
In particular,
we have $\kappa_{\rm 1}=26$ under ERH.
\end{corollary}

\begin{remark}{\rm
In Corollary~\ref{Konelist},
possible failure examples
contradicting ERH have class group $C_3^r$ for $r\geq 5$
by Watkins' result \cite{Watkins}.
}
\end{remark}

Using the set $I_1$ from Corollary~\ref{Konelist} we compute the family $\mathcal{K}_{2a}$ as prescribed in the previous section. 
\begin{lemma}\label{lem2a}
\label{TwoA}
We have $\kappa_{\rm 2a}=307$ under ERH.
\end{lemma}

Next consider the set $R_{p^*}$ for $\mathcal{K}_{\rm 2b}$. For $u=3$ we describe a shortcut to compute this set. We could also apply the algorithm from Section \ref{lastsec}.
It is enough for computing $R_{p^*}$ 
to find all the imaginary quadratic fields $k$
with $E(k)\mid 2u$.
However,
all such fields are not determined yet even assuming ERH
\cite{EKN}.
For each $p^*\in I_1$ with $|p^*|\leq 4027$
we can obtain
all the imaginary quadratic fields $k=\Q(\sqrt{p^*q^*})$
such that $q^*\in P^*_+$ and $E(\Q(\sqrt{p^*q^*}))\mid 2u$.
Indeed, fortunately,
the Sylow $2$-subgroup of $\Cl(\Q(\sqrt{p^*q^*}))$
is small and has a tractable generator
in our situation.

\begin{lemma}[Boyd and Kisilevsky {\cite[Lemma 1]{BoydKisi}}]
\label{BoydKisi}
Let $k$ be an imaginary quadratic field of discriminant $D$.
If $\alpha\in k$ is an algebraic integer with  $\alpha\not\in\Z$,
then $N_{k/\Q}(\alpha)\geq |D|/4$
where $N_{k/\Q}$ is the norm map from $k$ to $\Q$.
\end{lemma}

\begin{theorem}[Bach and Sorenson
{\cite[Theorem 5.1 and Table 3]{BachSore}}]
\label{BachSore}
Let $k$ be a quadratic field  of discriminant $D$
such that $|D|>e^{25}\approx 7.2\times 10^{10}$.
Assume ERH.
Then there exists a prime number $\ell$ splitting in $k$
such that
$\ell\leq (1.881\log(|D|) + 0.34\cdot 2 + 5.5)^2$.
\end{theorem}

\begin{theorem}[Elsenhans, Kl\"uners and Nicolae
{\cite[Theorem 1]{EKN}}]
\label{EKNexpsix}
The maximal absolute value $|D|$ of the discriminants $D$
of imaginary quadratic fields with class group
of exponent dividing $6$
under the condition $|D|<3.1\times 10^{20}$ is $|D|=5761140$.
\end{theorem}

\begin{lemma}
\label{RpBound}
Assume that $p^*\in I_1$ and $q^*\in R_{p^*}$ with $|p^*|\leq 4027$.
Then, under ERH, $q^*$ is less than $5761140/|p^*|$.
\end{lemma}

\begin{proof}
Let $p^*$ be in $I$ with $|p^*|\leq 4027$, and $q^*$ in $R_{p^*}$.
Then $k=\Q(\sqrt{p^*q^*})$ satisfies $\Cl(k)\simeq C_3^r\times C_2$
for some rational integer $r\geq 0$.
Here the discriminant $D$ of $k$ is equal to $p^*q^*$.
For finding an upper bound for $|D|$ we may assume that
$|D|>e^{25}\approx 7.2\times 10^{10}$.
Let $\mathfrak{p}$ be the unique prime ideal of $k$ above $p$.
It follows from genus theory that
$\mathfrak{p}$ is of order $2$ in $\Cl(k)$.
Indeed, $\mathfrak{p}^2=(p)$ and
$\mathfrak{p}$ is not principal for $q^*\geq 5$.
Theorem~\ref{BachSore} guarantees that
there exists a prime number $\ell$ splitting in $k$
such that
$\ell\leq f_1(|D|)$
where $f_1(x)$ is defined by
\begin{center}
$f_1(x)=(1.881\log x + 0.34\cdot 2 + 5.5)^2$
\end{center}
for $x>0$.
Let $\mathfrak{l}$ be a prime ideal of $k$ above such an $\ell$.
By the assumption on the structure on $\Cl(k)$,
the ideal $\mathfrak{l}^3$ is of order $1$ or $2$ in $\Cl(k)$.
The group $C_3^r\times C_2$ has a unique element of order $2$.
Thus $\mathfrak{l}^3$ or $\mathfrak{p}\mathfrak{l}^3$ is principal.
There exists an algebraic integer $\alpha$ in $k$
such that $\mathfrak{l}^3=(\alpha)$
or $\mathfrak{p}\mathfrak{l}^3=(\alpha)$.
Here $\mathfrak{l}\mid \alpha$ and $\ell\nmid \alpha$,
which means that $\alpha\not\in\Z$.
Lemma~\ref{BoydKisi} implies that $N_{k/\Q}(\alpha)\geq|D|/4$. 
Thus $\ell^3$ or $p\ell^3$ is not less than $|D|/4$,
in particular, $p\ell^3\geq |D|/4$.
Hence we have $\ell\geq \root{3}\of{|D|/(4p)}
\geq f_2(|D|)$
where $f_2$ is defined by
\begin{center}
$f_2(x)=\root{3}\of{x/(4\cdot 4027)}$
\end{center}
for $x>0$.
Considering the derivatives
of $\sqrt{f_1(x)}$ and $\sqrt{f_2(x)}$
and using a calculator,
we can see that
$\max\{\rho\in\R\,\mid\,f_1(\rho)\geq f_2(\rho)\}
\leq 2.4\times10^{15}$.
Therefore it follows that $|D|\leq 2.4\times10^{15}$ since otherwise such an $\ell$ exists, which proves different exponent.
Since the bound $3.1\times 10^{20}$ of Theorem~\ref{EKNexpsix}
is greater than $2.4\times10^{15}$,
Theorem~\ref{EKNexpsix} yields $|D|\leq 5761140$ under ERH.
Hence we have $q^*=|D|/|p^*|\leq 5761140/|p^*|$.
\end{proof}

\begin{remark}
\label{relationtoEKN}
{\rm
The method used in the proof of Lemma~\ref{RpBound} above
is a refinement of that in \cite{EKN}.
One has that
$5761140=2^2\cdot 3\cdot 5\cdot 7\cdot 11\cdot 29\cdot 43$
and $\Cl(\Q(\sqrt{-5761140}))\simeq C_6^2\times C_2^4$.
}
\end{remark}

For each $p^*\in I_1$ and $q^*\in R_{p^*}$,
we can check whether $E(\Q(\sqrt{p^*},\sqrt{q^*}))\mid u$
or not, due to Theorems~\ref{FroehlichOddCri} and \ref{LemmerOdd} as described in the last section.

\begin{lemma}
\label{TwoB}
We have $\kappa_{\rm 2b}=58$ under ERH.
\end{lemma}

The table of numbers of fields in
$\mathcal{K}_{\rm 2a}$
and $\mathcal{K}_{\rm 2b}$
classified by rank is as follows.
For example, the family $\mathcal{K}_{\rm 2b}$
has $12$ fields $K$ such that $\Cl(K)\simeq C_3^r$ with $r=2$.
\begin{center}
$\begin{array}{|c|r|r|r|r|r|r|}
\hline
\text{family}
&r=0&r=1&r=2&r=3&r=4&\text{total}\\ \hline
\mathcal{K}_{\rm 2a}&
32& 133& 110& 31& 1& 307\\ \hline
\mathcal{K}_{\rm 2b}&
15& 30& 12& 1& 0& 58\\ \hline
\text{total}&
47& 163& 122& 32& 1& 365\\ \hline
\end{array}$
\end{center}
The maximal discriminants of $K$ in $\mathcal{K}_{\rm 2a}$
with $\Cl(K)\simeq C_3,C_3^2,C_3^3$ and $C_3^4$ are
\begin{center}
\begin{tabular}{l}
\bimax{20715557041}{163^2  883^2}{-163}{-883},\\%C3
\bimax{430862272801}{163^2  4027^2}{-163}{-4027},\\%C3 x C3
\bimax{13340675895121}{907^2  4027^2}{-907}{-4027} and\\%C3 x C3 x C3
\bimax{6704790388321}{643^2  4027^2}{-643}{-4027}, \\%C3 x C3 x C3 x C3
\end{tabular}
\end{center}
respectively.
The maximal discriminants of $K$ in $\mathcal{K}_{\rm 2b}$
with $\Cl(K)\simeq C_3,C_3^2$ and $C_3^3$ are
\begin{center}
\begin{tabular}{l}
\bimax{7767369}{3^2  929^2}{-3}{929},\\%C3
\bimax{1157836729}{7^2  4861^2}{-7}{4861} and\\%C3 x C3
\bimax{503688249}{3^2  7481^2}{-3}{7481},\\%C3 x C3 x C3
\end{tabular}
\end{center}
respectively.

Let us determine
$\mathcal{K}_{\rm 3}$ under ERH.
Using
$\mathcal{K}_{\rm 2}$,
we can obtain
$\mathcal{K}_{\rm 3}$ as described in the previous section.
\begin{lemma}\label{lem_u_3}
We have 
$\kappa_{\rm 3a}=35$ and $\kappa_{\rm 3b}=42$ under ERH.
\end{lemma}

The table of numbers of fields in
$\mathcal{K}_{\rm 3a}$ and $\mathcal{K}_{\rm 3b}$
classified by rank is as follows.
For example, $\mathcal{K}_{\rm 3b}$
has $12$ fields $K$ such that $\Cl(K)\simeq C_3^r$
with $r=2$.
\begin{center}
$\begin{array}{|c|r|r|r|r|r|r|}
\hline
\text{family}
&r=0&r=1&r=2&r=3&r=4&\text{total}\\ \hline
\mathcal{K}_{\rm 3a}&
8& 8& 17& 2& 0& 35\\ \hline
\mathcal{K}_{\rm 3b}&
9& 15& 12& 5& 1& 42\\ \hline
\text{total}&
17& 23& 29& 7& 1& 77\\ \hline
\end{array}$
\end{center}
The maximal discriminants of $K$ in $\mathcal{K}_{\rm 3a}$
with $\Cl(K)\simeq C_3,C_3^2$ and $C_3^3$ are
\begin{center}
\begin{tabular}{l}
\trimax{1048870932736}{2^8  11^4  23^4}{-4}{-11}{-23},\\%C3
\trimax{3918727948865536}{2^{12}  23^4  43^4}{-8}{-23}{-43} and\\%C3 x C3
\trimax{13142294978742801}{3^4  43^4  83^4}{-3}{-43}{-83},\\%C3 x C3 x C3
\end{tabular}
\end{center}
respectively.
The maximal discriminants of $K$ in $\mathcal{K}_{\rm 3b}$
with $\Cl(K)\simeq C_3,C_3^2$, $C_3^3$ and $C_3^4$ are
\begin{center}
\begin{tabular}{l}
\trimax{63330416557681}{7^4  13^4  31^4}{-7}{-31}{13},\\%C3
\trimax{7860782851035136}{2^{12}  11^4  107^4}{-11}{-107}{8},\\%C3 x C3
\trimax{69969611497148416}{2^{12}  19^4  107^4	}{-19}{-107}{8} and\\%C3 x C3 x C3
\trimax{6505835909336928256}{2^{12}  59^4  107^4}{-59}{-107}{8},\\%C3 x C3 x C3 x C3
\end{tabular}
\end{center}
respectively.

\section{Computation for the case $u=5$ under ERH}

In this section we fix $u=5$.
 We remark that the hard part is to compute the families $\mathcal{K}_{1}$ and
$\mathcal{K}_{2b}$, i.e. the imaginary quadratic fields of exponent 5 and the imaginary quadratic fields with class group of type $C_2\times C_5^r$ for some
$r\geq 0$. The latter fields are described in the sets $R_p^*$. We will describe this computation in more detail in the next section.
Like Corollary \ref{Konelist} we have

\begin{theorem}[Elsenhans, Kl\"uners and Nicolae
{\cite[Theorems 1 and 2]{EKN}}]
\label{KonelistFive}
Under ERH, the following is the list of all
squarefree rational integers $d>0$
such that $\Cl(\Q(\sqrt{-d}))\simeq C_5^r$
for some rational integers $r\geq 0$.
\begin{center}
$\begin{array}{|c|l|c|}
\hline
r& d & \sharp\\
\hline
0&1, 2, 3, 7, 11, 19, 43, 67, 163 & 9\\ \hline
1&\begin{array}{@{}l@{}}
47, 79, 103, 127, 131, 179, 227, 347, 443, 523,
 571, 619, 683, 691, 739,\\
  787, 947, 1051, 1123, 1723,
 1747, 1867, 2203, 2347, 2683\\
\end{array} & 25\\ \hline
2&12451, 37363 & 2\\ \hline
\end{array}$
\end{center}
Here the numbers in the right column $\sharp$
are the numbers of $d$'s in the respective row.
In particular,
we have $\kappa_{\rm 1}=36$ under ERH.
\end{theorem}

\begin{remark}{\rm
In Theorem~\ref{KonelistFive},
possible failure examples
contradicting ERH have class group $C_5^r$ for $r\geq 3$
by Watkins' result \cite{Watkins}.
}
\end{remark}
 Like in Lemma \ref{lem2a} we can prove
\begin{lemma}\label{TwoAFive}
We have $\kappa_{\rm 2a}=537$ under ERH.
\end{lemma}

The difficult task will be to determine the sets $R_{p^*}$. We give some details of this computation later on.
\begin{center}
$\begin{array}{|@{\,}c@{\,}|@{\,}l@{\,}|%@{\,}r@{\,}|
}
\hline
p^*&R_{p^*}%&\sharp\ 
\\
\hline
-3 &
\begin{array}{@{}l@{}}
5, 8, 17, 41, 53, 89, 101, 233, 569, 593, 641, 881,\\
 1049, 1097, 1361, 1409, 1721, 1889, 11177, 38729% 20
\\
\end{array}
\\
\hline
-4 &
\begin{array}{@{}l@{}}
5, 13, 37, 181, 197, 229, 317, 373, 421, 541, 613,\\
 709, 757, 853, 877, 1093, 1213, 22717% 18
\\
\end{array}
\\
\hline
-7 &5, 13, 17, 61, 661, 733, 829, 1069, 9973, 12157, 21661% 11
\\
\hline
-8& 5, 29, 37, 61, 109, 173, 197, 269, 7829, 9461, 19301% 11
\\
\hline
-11& 8, 13, 17, 29, 73, 281, 569, 953% 8
\\
\hline
-19& 41, 193, 241, 337, 433% 5
\\
\hline
-43& 8, 73, 929% 3
\\
\hline
-47& 5, 13% 2
\\
\hline
-103& 37% 1
\\
\hline
-127& 5, 29, 109% 3
\\
\hline
-227& 17, 41% 2
\\
\hline
-347& 8% 1
\\
\hline
\end{array}$
\end{center}

Here
$R_{p^*}=\emptyset$
for $24$ numbers $p^*\in I_1$ which do not appear
in the above table.

%Now we have an assumption
%\begin{center}
%(*)\quad $R_{p^*}=R'_{p^*}$ for every $p^*\in I_1$.
%\end{center}
%One can check that
%$R'_{p^*}$ is equal to 
%the subset of $R_{p^*}$ consisting of integers $q^*$
%which are found in Malle's database \cite{Malle}.
%One has the same result by using the other database \cite{LMFDB}.

\begin{lemma}
\label{TwoBFive}
We have $\kappa_{\rm 2b}=82$ under ERH. % and (*).
\end{lemma}

The table of the numbers of fields in
$\mathcal{K}_{\rm 2a}$
and $\mathcal{K}_{\rm 2b}$ %$\mathcal{K}'_{\rm 2b}$
classified by rank is as follows.
%where
%$\mathcal{K}_{\rm 2b} %\mathcal{K}'_{\rm 2b}
%=\{\Q(\sqrt{p^*},\sqrt{q^*})\in\mathcal{K}_{\rm 2b}
%\,\vert\,p^*\in I_1,q^*\in R_{p^*} %R'_{p^*}
%\}$.
For example, the family $\mathcal{K}_{\rm 2b}$ %$\mathcal{K}'_{\rm 2b}$
has $18$ fields $K$ such that $\Cl(K)\simeq C_5^r$ with $r=2$.

\begin{center}
$\begin{array}{|c|r|r|r|r|r|r|}
\hline
\text{family}
&r=0&r=1&r=2&r=3&r=4&\text{total}\\ \hline
\mathcal{K}_{\rm 2a}&
32& 194& 256& 54& 1& 537\\ \hline
\mathcal{K}_{\rm 2b}&
%\mathcal{K}'_{\rm 2b}&
15& 49& 18& 0& 0& 82\\ \hline
\text{total}&
47& 243& 274& 54& 1& 619\\ \hline
\end{array}$
\end{center}
The maximal discriminants of $K$ in $\mathcal{K}_{\rm 2a}$
with $\Cl(K)\simeq C_5,C_5^2,C_5^3$ and $C_5^4$ are
\begin{center}
\begin{tabular}{l}
%[191256654241, 39652221594001, 10049045790215041, 216417285820264369]
\bimax{191256654241}{163^2  2683^2}{-163}{-2683},\\%C5
\bimax{39652221594001	}{2347^2  2683^2}{-2347}{-2683},\\%C5 x C5
\bimax{10049045790215041}{2683^2  37363^2}{-2683}{-37363} and\\%C5 x C5 x C5
\bimax{216417285820264369}{12451^2  37363^2}{-12451}{-37363},\\%C5 x C5 x C5 x C5
\end{tabular}
\end{center}
respectively.
The maximal discriminants of $K$ in $\mathcal{K}_{\rm 2b}$
%$\mathcal{K}'_{\rm 2b}$
with $\Cl(K)\simeq C_5$ and $C_5^2$ are
\begin{center}
\begin{tabular}{l}
%[109893289, 23841830464, 0, 0]
\bimax{109893289}{11^2  953^2}{-11}{953} and\\%C5
\bimax{23841830464}{2^6  19301^2	}{-8}{19301},\\%C5 x C5
\end{tabular}
\end{center}
respectively.

Let 
$\mathcal{K}_{\rm 3a}$
and 
$\mathcal{K}_{\rm 3b}$
denote the subfamilies of $\mathcal{K}_{\rm 3}$ such that
\begin{center}
$\begin{array}{r@{\,}l}
\mathcal{K}_{\rm 3a}
=&\{\Q(\sqrt{p_1^*},\sqrt{p_2^*},\sqrt{p_3^*})\in\mathcal{K}_{\rm 3}
\,\mid\,p_1^*,p_2^*,p_3^*\in P^*_-\},\\
\mathcal{K}_{\rm 3b}
=&\{\Q(\sqrt{p_1^*},\sqrt{p_2^*},\sqrt{p_3^*})\in\mathcal{K}_{\rm 3}
\,\mid\,p_1^*,p_2^*\in P^*_-, p_3^{*}\in P^*_+\}\\
\end{array}$
\end{center}
with cardinalities $\kappa_{\rm 3a}$ and $\kappa_{\rm 3b}$, respectively. Similar to Lemma \ref{lem_u_3} we can prove
 \begin{lemma}
We have 
$\kappa_{\rm 3a}=29$ under ERH.
We have 
$\kappa_{\rm 3b}=39$ under ERH. % and (*).
\end{lemma}

The table of the numbers of fields in
$\mathcal{K}_{\rm 3a}$ and $\mathcal{K}_{\rm 3b}$
%$\mathcal{K}'_{\rm 3b}$
classified by rank is as follows. 
%where
%$\mathcal{K}'_{\rm 3b}
%=\{\Q(\sqrt{p_1^*},\sqrt{p_2^*},\sqrt{p_3^*})\in\mathcal{K}_{\rm 3b}
%\,\vert\,p_1^*,p_2^*\in I_1,q^*\in R'_{p_1^*}\cap R'_{p_2^*}\}$.
For example, $\mathcal{K}_{\rm 3b}$ %$\mathcal{K}'_{\rm 3b}$
has $12$ fields $K$ such that $\Cl(K)\simeq C_5^r$
with $r=2$.

\begin{center}
$\begin{array}{|c|r|r|r|r|r|r|}
\hline
\text{family}
&r=0&r=1&r=2&r=3&r=4&\text{total}\\ \hline
\mathcal{K}_{\rm 3a}&
8& 6& 14& 1& 0& 29\\ \hline
\mathcal{K}_{\rm 3b}&
%\mathcal{K}'_{\rm 3b}&
9& 12& 12& 5& 1& 39\\ \hline
\text{total}&
17& 18& 26& 6& 1& 68\\ \hline
\end{array}$
\end{center}

The maximal discriminants of $K$ in $\mathcal{K}_{\rm 3a}$
with $\Cl(K)\simeq C_5,C_5^2$ and $C_5^3$ are
\begin{center}
\begin{tabular}{l}
%[672285245399296, 923878543897348081, 102823251817101201, 0]
\trimax{672285245399296}{2^8  19^4  67^4}{-4}{-19}{-67},\\%C5
\trimax{923878543897348081}{7^4  43^4  103^4}{-7}{-43}{-103} and\\%C5 x C5
\trimax{102823251817101201}{3^4  47^4  127^4}{-3}{-47}{-127},\\%C5 x C5 x C5
\end{tabular}
\end{center}
respectively.
The maximal discriminants of $K$ in $\mathcal{K}'_{\rm 3b}$
with $\Cl(K)\simeq C_5,C_5^2$, $C_5^3$ and $C_5^4$ are
\begin{center}
\begin{tabular}{@{}l@{}}
%[136166867931136, 3246907040793595201, 977807264466179992321, 2693876092569442561]
\trimax{136166867931136}{2^{12}  7^4  61^4}{-7}{-8}{61},\\%C5
\trimax{3246907040793595201}{11^4  17^4  227^4}{-11}{-227}{17},\\%C5 x C5
\trimax{977807264466179992321}{19^4  41^4  227^4}{-19}{-227}{41} and\\%C5 x C5 x C5
\trimax{2693876092569442561}{11^4  29^4  127^4}{-11}{-127}{29},\\%C5 x C5 x C5 x C5
\end{tabular}
\end{center}
respectively.

\section{Computation of $R_{p^*}$ and the family ${\mathcal K_{2b}}$}\label{lastsec}
In this section we describe how to compute the set  ${\mathcal{K}_{2b}}$ for odd prime exponents $u$. In the following we assume that
$p^*\in I_1$ is given with $p^*<0$ such that the exponent of the class group of $\Q(\sqrt{p^*})$ divides $u$. We would like to compute the set $R_{p^*}$ defined in \eqref{Rpstar} consisting of positive fundamental
prime (power) discriminants $q$ such that the exponent of the class group of $L:=\Q(\sqrt{p^*q})$ divides $2u$. Note that the 2-part of $\Cl(L)$ is cyclic. Therefore we can ignore all fields $L$, for which $4$ divides the order of the class group. 
\begin{lemma}\label{rank4}
	Let $L:=\Q(\sqrt{p^*q})$  with $p,q$ prime, $q\equiv 1 \bmod 4$ and $p^*<0$. Then 
	$4 \nmid |\Cl(L)|$ if and only if one of the following cases happens:
	\begin{enumerate}
		\item $p$ odd and $\legendre{q}{p}=\legendre{-p}{q}=-1$,
		\item $p=2$ and $q\equiv 5 \bmod 8$.
	\end{enumerate}
\end{lemma}
\begin{proof}
	 By Redei's criterion, see \cite{Red}, the class group has an element of order 4, if $\legendre{q}{p}=1$ or $q\equiv 1\bmod 8$, respectively.
\end{proof}
Let $\ell$ be the smallest split
prime in $L$. Then the order of a prime ideal above $\ell$ is either $u$ or $2u$ which means that $\ell^u$ or $\ell^{2u}$ is a norm, respectively. If the order
is $2u$ then $p\ell^u$ is a norm.
Using Lemma \ref{BoydKisi} this means that we have a solution of the following norm equation:
\begin{eqnarray}\label{norm2E}
	X^2 -p^*q Y^2 =& 4\ell^up   & \text{ if the order is }2u\\
	X^2 -p^*q Y^2 =& 4\ell^u    & \text{ if the order is }u.\label{normE}
\end{eqnarray}
Note that $p^*<0$ and $p^*=-p$ except for $p=2$ we have $p^*=-4$ or $p^*=-8$. 
The easy description of our algorithm is that it checks all $X\leq \sqrt{4\ell^up}$ if
$\sqrt{X^2-4\ell^up}$ corresponds to a field, in which $\ell$ is the smallest splitting prime.
This algorithm is very similar to the one described in \cite[Lemma 19, Remark 20]{EKN}. In the following we would like to reduce the number of $X$ to consider.  Therefore we distinguish six cases, depending on whether we have an element of order $u$ or $2u$ and if $p^*=-4,-8$ or $p^*$ is odd.  In these situations we can simplify the equations to the following equations (with new $X$ and $Y$).
We deal with $q=2$ or $\ell= 2$ separately. Note that the following theorem reduces the number of possibilities to consider dramatically. E.g. assume that $p$ is odd. Then using \eqref{norm2E} and \eqref{normE} we have to check two cases. Using the Legendre symbol condition we know that we only need to check one case. If we are in case (1) we know that $X$ is divisible by $p$ which gives a big reduction of the number of possibilities to consider. In case (4) we get some parity condition for $X$.
 
\begin{theorem} \label{thm:cases}
	Let $\ell>2$ be a splitting prime in $L:=\Q(\sqrt{p^*q})$ for $q\equiv 1 \bmod 4$ such that the ideal above $\ell$ has order $o$ in the class group of $L$. Then in each case there exists one equation with solution $(X,Y)\in\Z^2$, i.e. at least one line is true:
	\begin{enumerate}
		\item Assume $p\equiv 3 \bmod 4$, $o=2u$. Then we have $\legendre{\ell}{p}= -1$ and
	\begin{eqnarray*}
	(pX)^2 - 4\ell^up = -pq Y^2  & & X\equiv  1 \bmod 2 \\
	(pX)^2 - \ell^up = -pq Y^2   & & X\equiv 1 \bmod 2 \text{ for }  \ell \equiv 3 \bmod 4   \\
	&& X \equiv 2 \bmod 4 \text{ for } \ell p \equiv 7 \bmod 8 \\
	&& X \equiv 0 \bmod 4 \text{ for } \ell p \equiv 3 \bmod 8 \\
\end{eqnarray*}
\item Assume $p^*=-8$ and $o=2u$. Then
\begin{eqnarray*}
	X^2 - \ell^up = -pq Y^2 && X \equiv 2 \bmod 4 \text{ for } \ell  \equiv 7 \bmod 8 \\
	&& X \equiv 0 \bmod 4 \text{ for } \ell \equiv 5 \bmod 8 \\
\end{eqnarray*}
\item Assume $p^*=-4$ and $o=2u$. Then
\begin{eqnarray*}
	X^2-2\ell^u=-qY^2 && X \equiv 1 \bmod 2 \text{ for }\ell \equiv 3 \bmod 4 \\
\end{eqnarray*}
\item Assume $p\equiv 3 \bmod 4$, $o=u$. Then we have $\legendre{\ell}{p}= 1$ and
\begin{eqnarray*}
	X^2-4\ell^u=-pqY^2 && X \equiv 1 \bmod 2 \\
	X^2 - \ell^u = -pq Y^2   & & X\equiv 1 \bmod 2 \text{ for }  \ell \equiv 1 \bmod 4  \\
	&& X \equiv 2 \bmod 4 \text{ for } \ell  \equiv 7 \bmod 8 \\
	&& X \equiv 0 \bmod 4 \text{ for } \ell  \equiv 3 \bmod 8 \\
\end{eqnarray*}
\item Assume $p^*=-8$ and $o=u$. Then
\begin{eqnarray*}
	X^2 - \ell^u = -2qY^2   & & X\equiv 1 \bmod 2 \text{ for }  \ell \equiv 1,3 \bmod 8 \\
	&& 
\end{eqnarray*}
\item Assume $p^*=-4$ and $o=u$. Then
\begin{eqnarray*}
	X^2 - \ell^u = -qY^2   & & X\equiv 1 \bmod 2 \text{ for }  \ell \equiv 1 \bmod 4  \\
	&& X \equiv 2 \bmod 4 \text{ for } \ell  \equiv 1 \bmod 8 \\
	&& X \equiv 0 \bmod 4 \text{ for } \ell  \equiv 5 \bmod 8 \\
\end{eqnarray*}
\end{enumerate}
\end{theorem}
\begin{proof}
	We use \eqref{norm2E} of the previous page for the case $2u$.   In all cases we see that $X$ must be divisible by $p$ (even true for $p=2$) and we can replace $X$ by $pX$
	to get the new equation $(pX)^2 -p^*q Y^2 = 4\ell^up$ and equivalently $pX^2 -p^*/p qY^2=4\ell^u$. For odd $p$ we arrive at the first equation
	and we get the identity
	$$\legendre{\ell}{p} = \legendre{4\ell^u}{p}= \legendre{p^2+qY^2}{p} = \legendre{q}{p} =-1 $$ by Lemma \ref{rank4}. 

In the exponent $u$-case we have
$$\legendre{\ell}{p} = \legendre{4\ell^u}{p} = \legendre{X^2+pqY^2}{p} = \legendre{X^2}{p} =1.$$
The proof of the odd $p$ cases uses the fact that $-pq\equiv 5 \bmod 8$ since otherwise 2 is split or ramified. In the cases $p^*=-4,-8$ we see that $X,Y$ must be divisible by $2$ and we can divide the whole equation by $4$. Finally we use Lemma \ref{rank4} to see that $q\equiv 5 \bmod 8$ which gives further conditions.
\end{proof}

We have to check for all $p^*\in I_1$ if we find a suitable $q\in P^*_+$. For a given $p^*$ it is easy to check $q=2$ by hand. Therefore we can assume that $q=q^* \equiv 1 \bmod 4$. Our implementation is slightly different for the cases $p^*=-4, p^*=-8$, or $p^*=-p$ for $p\equiv 3 \bmod 4$. Let us only describe the case that $p^*$ is odd. We would like to use Theorem \ref{thm:cases}. Therefore we need to check the case $\ell=2$ by hand which means that using \eqref{norm2E} and \eqref{normE} we get:
\[
	X^2 - 4\ell^up  = p^*q Y^2 
	\text{ or }
X^2 -4\ell^u = p^*q Y^2.
\]
In the first case we check for all $X < \sqrt{4p\cdot 2^u}$
if $\Q(\sqrt{X^2-4p2^u})$ has class group isomorphic to $C_2\times C_u^r$. We also need to
test for all $X < \sqrt{4\cdot 2^u}$
if $\Q(\sqrt{X^2-4\cdot 2^u})$ has class group isomorphic to $C_2\times C_u^r$.

Now we start a loop over all odd primes $\ell$. For each $\ell$ we first check using Theorem \ref{BachSore} which assumes ERH, if $\ell$ can be the smallest splitting prime for a field with discriminant larger than $4\ell^u p$. If not we can skip this $\ell$ and all larger primes.

Now we compute $\legendre{\ell}{p}$ and decide, if we are in case (1) or (4) of Theorem \ref{thm:cases}. Suppose we are in case (1), the other case treated similarly.
Now compute for all odd $X$ such that $(pX)^2 < 4\ell^up$ the value
\[D:=(pX)^2-4\ell^up. \]
Note that $D$ is equal to $-pq$ up to a square, if we check a valid $X$. 
Now we check, if $\ell$ is the smallest splitting prime of $\Q(\sqrt{D})$ by checking for all primes
$q<\ell$ if $\legendre{D}{q}\ne 1$ and some extra test if $q\mid D$. If one Legendre symbol equals 1 then we can skip this $D$. Otherwise we factorize $D$ and finally check the class group.

Depending on the congruence we might have to check one of the other possible cases in (1).

We implemented the search in the computer algebra system Hecke \cite{FHHJ} which is based on the Julia language \cite{julia}. Our code can be downloaded from the web page 
https://math.uni-paderborn.de/en/ag/ca/research/exponent/. The mathematics of the code is already explained, but we used a lot of small details which makes the code more efficient. We do not try to explain the code here. On the same site we have put the list of computed fields. We remark that all the fields $L=\Q(\sqrt{p^*q})$ with exponent dividing 6 or 10 can be found in the databases \cite{LMFDB} and \cite{Malle}. Our computation shows that there are no further fields
with this property beyond the computed range of these databases.

\section*{Acknowledgment}
The research was done during a sabbatical of the second author at Paderborn University.
%%%%%

\end{document}